\documentclass[graybox,vecarrow]{svmult}

\usepackage{mathptmx}       
\usepackage{helvet}         
\usepackage{courier}        
\usepackage{makeidx}         
\usepackage{graphicx}        
\usepackage{multicol}        
\usepackage[bottom]{footmisc}

\usepackage{amsmath}  
     
\usepackage{amssymb,ifsym}  
 
\usepackage{enumerate}
\usepackage{stmaryrd}
\usepackage{url} 

\usepackage[all]{xy}

\usepackage{esvect}      

\usepackage{scalerel}

   \usepackage{bbding}        

\usepackage{mathrsfs}       

\usepackage[round]{natbib}       

\smartqed

\DeclareSymbolFont{symbolsC}{U}{txsyc}{m}{n}
\DeclareMathSymbol{\strictif}{\mathrel}{symbolsC}{74}              

\newcommand{\CC}{\mathcal{C}}

\def\om{\omega}

\def\bo{\mathop{\scalerel*{\Box}{gX}}}           
\def\di{\mathord{\scalerel*{\Diamond}{gX}}}     

\def\C{\mathbb{C}}

\def\cl{\mathop{\mathsf{Cl}}}   

\def\de{\mathop{\mathsf{De}}}   

\def\D{\mathbb{D}}

\def\F{\mathcal{F}}

\renewcommand{\ge}{\geqslant}       
\renewcommand{\le}{\leqslant} 

\def\int{\mathop{\sf Int}} 

\def\inv{^{-1}}

\def\M{\mathcal{M}}

\def\P{\mathbb{P}}

\def\ph{\varphi}

\def\<{\langle}
\def\>{\rangle}

\def\sub{\subseteq}

\def\Vec#1#2{#1_0,\ldots{},#1_{#2}}

\def\CC{\mathscr{C}}

\def\up{^\uparrow}
 
\begin{document}

\title*{Modal Logics of Some Hereditarily Irresolvable Spaces}
\author{Robert Goldblatt}
\institute{Robert Goldblatt \at Victoria University of Wellington, \email{rob.goldblatt@msor.vuw.ac.nz} }
\maketitle


\abstract{
A topological space is \emph{hereditarily $k$-irresolvable} if none of its subspaces can be partitioned into $k$ dense subsets, We use this notion  to provide a topological semantics for a sequence of modal logics whose $n$-th member K4$\mathbb{C}_n$ is characterised by validity in  transitive Kripke frames of circumference at most $n$. We show that under the interpretation of the modality $\Diamond$ as the derived set (of limit points) operation, K4$\mathbb{C}_n$ is characterised by validity in all spaces that are hereditarily $n+1$-irresolvable and have the T$_D$ separation property.
\newline\indent
We also identify the extensions of K4$\mathbb{C}_n$ that result when the class of spaces involved is restricted to those that are crowded, or densely discrete, or openly irresolvable, the latter meaning that every non-empty open subspace is  2-irresolvable. Finally we give a topological semantics for K4M, where M is the McKinsey axiom. 
}



\section{Introduction}      

One theme of this article is the use of \emph{geometric} ideas in the semantic analysis of logical systems.
Another is the use of \emph{relational} semantics in the style of Kripke.
Both themes feature prominently in the many-faceted research portfolio of Alasdair Urquhart.
In the area of relevant logic he discovered how to construct certain  relational models of relevant implication out of projective geometries, and related these to modular geometric lattices (\citeyear{urqu:rele83}, see also \citeyear{urqu:geom17}). This led to his striking demonstration in \citeyearpar{urqu:unde84} that the main systems of relevant implication are undecidable, and then to a proof in \citeyearpar{urqu:fail93} that these systems fail to satisfy the Craig interpolation theorem.
In the area of Kripke semantics for modal logics, in addition to an early book \citeyearpar{resc:temp71} with Rescher on temporal logics, his contributions have  included the construction in \citeyearpar{urqu:deci81} of a normal logic that is recursively axiomatisable and has the finite model property but is undecidable; and the proof in \citeyearpar{urqu:firs15} that there is a formula of quantified S5 which, unlike the situation of propositional S5, is not equivalent to any formula of modal degree one.
His article \citeyearpar{urqu:topo78} on topological representation of lattices has been very influential. It assigned to any bounded lattice a dual topological space with a double ordering, and showed that the original lattice can be embedded into a  lattice of certain `stable' subsets of its dual space. This generalised the topological representations of \cite{ston:repr36} for Boolean algebras and \cite{prie:repr70} for distributive lattices. It stimulated the development of further duality theories for lattices, such as those of  \cite{hart:topo92,hart:exte93},  \cite{allw:dual93}, \cite{plos:natu95} and \cite{hart:ston97}. The structure of Urquhart's dual spaces was exploited by \cite{allw:krip93} to develop a Kripke semantics for linear logic, and by \cite{dzik:rela06} to do likewise for non-distributive logics with various negation operations. Urquhart's article also played an important role in the development of the important notion of a \emph{canonical extension}. This was first introduced by \cite{jons:bool51} for Boolean algebras with operators, and  is closely related to the notion of  canonical model of a modal logic. After several decades of evolution, \cite{gehr:boun01} gave an axiomatic definition of a canonical extension of any bounded lattice-based algebra, showing that it is unique up to isomorphism. In proving that such an extension exists, they observed that it can be constructed as the embedding of the original algebra into the lattice of stable subsets of its Urquhart dual space.  \cite{crai:reco14} have clarified the relationship between Urquhart's construction and other manifestations of canonical extensions. This area has now undergone substantial development and application, by Gehrke and others, of theories of canonical extensions and duality for lattice expansions. There are surveys of this progress in \citep{gehr:cano18} and \citep{gold:cano18}. In \citep{gold:morp20} a new notion of bounded morphism between polarity structures has been introduced that provides a duality with homomorphisms between lattices with operators.
The seed that Urquhart planted has amply borne fruit.

\bigskip
The geometrical ideas of the present article come from topology. Our aim is to provide a topological semantics for a sequence  of modal logics that were originally defined by properties of their binary-relational Kripke models. The $n$-th member of this sequence is called 
K4$\C_n$, and is the smallest normal extension of the logic K4 that includes an axiom scheme $\C_n$ which will be described below. It was shown in \citep{gold:moda19} that the theorems of K4$\C_n$ are characterised by validity in all finite transitive Kripke frames that have circumference at most $n$, meaning that any cycle in the frame has length at most $n$, or equivalently that any non-degenerate cluster has at most $n$ elements. For $n\geq 1$, adding the reflexivity axiom $ \ph\to\di\ph$ to K4$\C_n$ gives the logic S4$\C_n$ which is characterised by validity in all finite reflexive transitive Kripke frames that have circumference at most $n$. It was also shown that S4$\C_n$ has a topological semantics in which any formula of the form $\di\ph$ is interpreted as the topological closure of the interpretation of $\ph$.
Under this interpretation, S4$\C_n$ is characterised by validity in all topological spaces that are 
\emph{hereditarily $n+1$-irresolvable}. 
Here a space is called  \emph{$k$-resolvable} if it can be partitioned into $k$ dense subsets, and is 
hereditarily $k$-irresolvable if none of its non-empty subspaces are $k$-resolvable. $k$ could be any cardinal, but we will deal only with finite $k$.

The interpretation of $\di$ as closure, now known as $C$-semantics, appears to have first been considered by \cite{tang:alge38}. \cite{mcki:solu41} showed that the formulas that are $C$-valid in all topological spaces are precisely the theorems of S4.
\cite{mcki:alge44,mcki:theo48} then undertook a much deeper analysis which showed that S4 is characterised by $C$-validity in any given space that has a certain `dissectability' property that is possessed by every finite-dimensional Euclidean space, and more generally by any
  metric space that is crowded, or dense-in-itself, i.e.\ has no isolated points.
They also suggested studying an alternative topological interpretation, now called  $d$-semantics, in which
$\di\ph$ is taken to be  the derived set, i.e.\ the set of limit points, of the interpretation of $\ph$. This $d$-semantics does not  validate the transitivity axiom $\di\di \ph\to\di\ph$. That  is $d$-valid in a given space iff the space satisfies the T$_D$  property that every derived set is closed.

The logic K4$\C_0$ is in fact the  G\"odel-L\"ob provability logic, which is known to have a topological characterisation  by $d$-validity in all spaces that are \emph{scattered}, meaning that every non-empty subspace has an isolated point.
The logic K4$\C_1$ was shown by  \cite{gabe:topo04} to be characterised by  $d$-validity in hereditarily 2-irresolvable spaces. Such spaces satisfy the T$_D$ property, which however may not hold in hereditarily $k$-irresolvable spaces for $k>2$. The principal new result proven here is that for  $n>1$, K4$\C_n$ is characterised by  $d$-validity in all spaces that are both hereditarily $n+1$-irresolvable and T$_D$. Our proof adapts the argument of  \citep{gabe:topo04} and also makes use of certain spaces that are  $n$-resolvable but not $n+1$-resolvable. These are provided by a construction of
\cite{elki:ultr69}.

Sections 2 and 3 review  the theory of relational and topological semantics for modal logic that we will be using. Section 4 is the heart of the article: it discusses hereditary irresolvability and gives the constructions that lead to our  characterisation of K4$\C_n$, Section 5 gives characterisations of some extensions of K4$\C_n$ that correspond to further topological constraints.

\section{Frames and Logics}
We begin with a review of relation semantics for propositional modal logic. A standard reference is \citep{blac:moda01}.
Formulas $\ph,\psi,\dots$ of  are constructed from some denumerable set of propositional variables by the  Boolean connectives $\top$,  $\bot$, $\neg$, $\land$, $\lor$, $\to$ and the unary modalities $\di$ and $\bo$. 
We write $\bo^*\ph$ as an abbreviation of the formula $\ph\land\bo\ph$, and $\di^*\ph$ for $\ph\lor\di\ph$.

A \emph{frame} $\F=(W,R)$ consists of a binary relation $R$ on a set $W$.
Each $x\in W$ has the set $R(x)=\{y\in W:xRy\}$ of \emph{$R$-successors}.
A \emph{ model} $\M=(W,R,V)$  on a frame has a valuation function $V$ assigning to each variable $p$ a subset $V(p)$ of $W$. The model then assigns to each formula $\ph$ a subset  $\M(\ph)$ of $W$, thought of as the set of points  at which $\ph$ is true. These truth sets are defined by induction on the formation of $\ph$, putting $\M(p)=V(p)$ for each variable $p$, interpreting each Boolean connective by the corresponding Boolean operation on subsets of $W$, and putting
\begin{align*}
\M(\di\ph) &= R^{-1}\M(\ph) =\{x\in W: R(x)\cap \M(\ph)\ne\emptyset\},
\\
\M(\bo\ph) &=\{x\in W: R(x)\sub \M(\ph)\}.
\end{align*}
A formula $\ph$ is \emph{true in model $\M$}, written  $\M\models\ph$, if 
$\M(\ph)=W$; and is \emph{valid in frame $\F$}, written $\F\models\ph$,  if it is true in all models on $\F$. 

A \emph{normal logic} is any set $L$ of formulas that
 includes all tautologies and all instances of the scheme
K:\enspace $\bo(\ph\to\psi)\to(\bo\ph\to\bo\psi)$,
and whose rules include  modus ponens and $\Box$-generalisation (from $\ph$ infer $\bo\ph$).  
The members of a logic $L$ may be referred to as the \emph{$L$-theorems}.
The set $L_\mathcal{C}$ of all formulas valid in (all members of) some given class $\mathcal{C}$ of frames is a normal logic. A logic is \emph{characterised by validity in $\mathcal{C}$}, or is \emph{sound and complete for validity in $\mathcal{C}$}, if it is equal to $L_\mathcal{C}$. The smallest normal logic, known as K, is characterised by validity  in all frames.

A logic is \emph{transitive}  if it contains all instances of the scheme
4:\enspace $\di\di\ph\to\di\ph$,  which is valid in precisely those frames that are transitive, i.e.\ their relation $R$ is transitive.
The smallest transitive normal logic,  known as K4, is characterised by validity  in all transitive frames.

In a transitive frame $\F=(W,R)$, a \emph{cluster} is a subset $C$ of $W$ that is an equivalence class under the equivalence relation
$
\{(x,y):x=y\text{ or } xRyRx\}.
$
The $R$-cluster containing $x$ is $C_x=\{x\}\cup\{y:xRyRx\}$.
If $x$ is irreflexive, i.e.\ not $xRx$, then $C_x=\{x\}$, and $C_x$ is called a \emph{degenerate} cluster.
Thus if $R$ is an irreflexive relation, then all clusters are degenerate.
If $xRx$, then $C_x$ is \emph{non-degenerate}. If $C$ is a non-degenerate cluster then it contains no irreflexive points and the relation $R$ is universal on $C$ and maximally so. 
A \emph{simple} cluster is non-degenerate with one element, i.e.\ a singleton $C_x=\{x\}$ with $xRx$.
If $R$ is \emph{antisymmetric}, i.e.\ $xRyRx$ implies $x=y$ in general, then every cluster is a singleton, so is either simple or degenerate.

The relation $R$ lifts to a well-defined relation on the set of clusters  by putting $C_xRC_y$ iff $xRy$. This relation is transitive and antisymmetric on the set of clusters. A cluster $C_x$ is \emph{final} if it is maximal in this ordering, i.e.\ there is no cluster $C\ne C_x$ with $C_xR C$. This is equivalent to requiring that $xRy$ implies $yRx$. 

We take the \emph{circumference} of a frame $\F$ to be the supremum of the set of all lengths of cycles in $\F$, where a cycle of length $n\geq 1$ is a finite sequence $x_1,\dots,x_{n}$ of distinct points such that $x_1R\cdots Rx_{n}Rx_1$.
The circumference is 0 iff there are no cycles.
In a transitive frame, the points of any cycle  are $R$-related to each other and are reflexive, and all belong to the same non-degenerate cluster. Conversely any finite non-empty subset of a non-degenerate cluster can be arranged (arbitrarily) into a cycle. Thus  if a finite transitive frame has  circumference $n\geq 1$, then $n$ is equal to the size of a largest non-degenerate cluster. The  circumference is $0$ iff the frame is irreflexive, i.e.\ has only degenerate clusters.
A  finite transitive frame has  circumference \emph{at most} $n$ iff each of its non-degenerate clusters has at most $n$ members.

Now given formulas $\Vec{\ph}{n}$,
define the formula $\P_n(\Vec{\ph}{n})$ to be
$$
\di(\ph_1\land\di(\ph_2\land\cdots \land\di(\ph_n\land\di \ph_0))\cdots)
$$
provided that $n\ge 1$. For the case $n=0$, put $\P_0(\ph_0)=\di\ph_0$.
Let $\D_n(\Vec{\ph}{n})$ be $\bigwedge_{ i<j\le n}\neg(\ph_i\land \ph_j)$ (for $n=0$ this is the empty conjunction $\top$\,).
Define $\C_n$ to be the scheme
$$
\bo\nolimits^*\D_n(\Vec{\ph}{n}) \to(\di \ph_0 \to
\di(\ph_0\land\neg\P_n(\Vec{\ph}{n}) ).
$$
Let K4$\C_n$ be the smallest transitive normal logic that includes the scheme $\C_n$. In
 \cite[Theorems 1 and 4]{gold:moda19}, the following were shown. 

\begin{theorem} \label{K4Cnfmp}
For all $n\geq 0$,
\begin{enumerate}[\rm(1)]
\item 
A transitive frame validates $\C_n$ iff it has circumference at most $n$ and no strictly ascending chains (i.e.\ no sequence $\{x_m:m<\om\}$ with $x_mRx_{m+1}$ but not  $x_{m+1}Rx_{m}$ for all $m$).
\item
A finite transitive frame validates $\C_n$ iff it has circumference at most $n$.
\item
A formula is a theorem of K4$\C_n$ iff it is valid in all finite transitive frames that have circumference at most $n$.
\qed
\end{enumerate}
\end{theorem}

The cases $n=0,1$ were described in detail in  \citep{gold:moda19}. K4$\C_0$ is equal to the  G\"odel-L\"ob modal logic
of provability, the smallest normal logic to contain the  L\"ob axiom 
$\bo(\bo \ph\to \ph)\to\bo \ph$. It is characterised by validity in all finite transitive frames that are irreflexive, i.e.\ all clusters are degenerate. This was first shown by \cite{sege:essa71}, and there have been a number of other proofs since (see \citealt[Chapter 5]{bool:logi93}).
 K4$\C_1$ is equal to the logic  K4Grz$_\Box$, where Grz$_\Box$ is the axiom
$$
\bo(\bo(p\to\bo p)\to p)\to \bo p.
$$
This logic was shown by \cite{gabe:topo04} to be characterised by validity in all finite frames whose clusters are all singletons (which may individually  be arbitrarily degenerate or simple).

Parts (1) and (3) of Theorem \ref{K4Cnfmp} imply that K4$\C_n$ is a \emph{subframe logic} in the sense of \cite{fine:logi85}
that it is complete for validity in its class of validating Kripke frames and this class is closed under subframes,

Part (2) of  of Theorem \ref{K4Cnfmp} states that validity of $\C_n$ exactly expresses the property of having circumference at most $n$ over the class of \emph{finite} transitive frames. But for transitive frames in general, there is no such formula, or even set of formulas:

\begin{theorem}
There does not exist any set of modal formulas whose conjoint validity in any transitive frame exactly captures the property of having circumference at most $n\geq 0$. 
\end{theorem} 
\begin{proof}
Let $\F_\om$ be the transitive irreflexive frame $(\om,<)$, and for each $1\leq m<\om$ let $\F_m$ be the frame consisting of a single non-degenerate $m$-element cluster on the set $\{0,\dots,m-1\}$. Let $f_m:\F_\om\to\F_m$ be the map 
$f_m(x)=x\bmod m$. Then $f_m$ is a surjective bounded morphism, a type of map that preserves validity of modal formulas in frames \cite[Theorem 3.14]{blac:moda01}. Hence any formula valid in $\F_\om$ will be valid in $\F_m$ for every $m\geq 1$.

Now suppose, for the sake of contradiction, that there does exist a set  $\Phi_n$ of formulas that is valid in a transitive frame iff that frame has circumference at most $n$. Then 
$\F_\om\models\Phi_n$, since $\F_\om$ has circumference 0, and therefore has circumference at most $n$, what ever $n\geq 0$ is chosen. But then by the previous paragraph, 
 $\F_{n+1}\models\Phi_n$, which contradicts the definition of $\Phi_n$ because  $\F_{n+1}$ has circumference $n+1$.
 \qed
\end{proof}

\section{Topological Semantics}

We first review some of the topological ideas that will be used to interpret modal formulas.
These can be found in many textbooks, such as \citep{will:gene70} and \citep{enge:gene77}.

 If $X$ is a topological space, we do not introduce a symbol for the topology of $X$, but simply refer to various subsets as being open or closed \emph{in $X$}.
If $x\in X$, then an \emph{open neighbourhood of $x$} is any open set that contains $x$. A subset of $X$ \emph{intersects} another subset if the two subsets have non-empty intersection. If $Y\sub X$, then $x$ is a \emph{closure point} of $Y$ if every open neighbourhood of $x$ intersects  $Y$. The set of all closure points of $Y$ is the \emph{closure} of $Y$, denoted $\cl_X Y$. It is the smallest closed superset of $Y$ in $X$. It has $Y\sub\cl_X Y=\cl_X \cl_X Y$.
A useful fact is that if $O$ is open in $X$, then $O\cap\cl_X Y\sub\cl_X(O\cap Y)$.
$Y$ is called \emph{dense in $X$} when $\cl_X Y=X$. This means that every non-empty open set intersects $Y$.

We write $\int_X Y$ for the interior of $Y$ in $X$, the largest open subset of $Y$. Thus $x\in \int_X Y$ iff some open neighbourhood of $x$ is included in $Y$.

Any subset $S$ of $X$ becomes a \emph{subspace} of $X$ under the topology whose open sets are all sets of the form $S\cap O$ with $O$ open in $X$. The closure operator $\cl_S$ of the subspace $S$ satisfies $\cl_S Y=S\cap \cl_X Y$. Hence a set $Y\sub S$ is dense in $S$, i.e.\ $\cl_S Y=S$, iff $S\sub\cl_X Y$.

A \emph{punctured neighbourhood of $x$} is any set of the form $O-\{x\}$ with $O$ an open neighbourhood of $x$.  If $Y\sub X$, then $x$ is a \emph{limit point} of $Y$ in $X$  if every punctured neighbourhood of $x$ intersects  $Y$.
The set of all limit points of $Y$ is the \emph{derived set} of $Y$, which we denote $\de_X Y$. (Other notations are $Y'$, $d Y$, and $\langle d\rangle Y$.)
In general, $\cl_XY=Y\cup\de_XY$.
 Another useful fact is that if $O$ is open in $X$, then $O\cap\de_X Y\sub\de_X(O\cap Y)$. If $S$ is a subspace of $X$, then any $Y\sub S$ has $\de_S Y=S\cap \de_X Y$.
 
 A \emph{T$_D$ space} is one in which the derived set $\de_X\{x\}$ of any singleton is closed, which is equivalent to requiring that  $\de_X\de_X\{x\}\sub\de_X\{x\}$.
 This in turn is equivalent to the requirement that any derived set $\de_X Y$ is closed, i.e.\  that $\de_X\de_X Y\sub \de_X Y$ for all $Y\sub X$  \cite[Theorem 5.1]{aull:sepa62}. The T$_D$ property is strictly weaker than the T$_1$ separation property that
 $\de_X\{x\}=\emptyset$ in general, which is itself equivalent to the requirement that any singleton is closed, and to the requirement that any two distinct points each have an open neighbourhood that excludes the other point.
 The simplest example of a non-T$_D$ space  is a two-element  space  $X=\{x,y\}$ with the \emph{indiscrete} (or \emph{trivial}\kern1pt) topology in which only $X$ and $\emptyset$ are open. It has $\de_X\{x\}=\{y\}$ and 
 $\de_X\{y\}=\{x\}$.
 A useful known fact is
 
 \begin{lemma}  \label{TDchar}
A space $X$ is T$_D$ iff it has $x\notin\cl_X\de_X\{x\}$ for all $x\in X$.
\end{lemma}
 \begin{proof}
Since $x\notin\de_X\{x\}$ in general, if $X$ is T$_D$ then $x\notin\cl_X\de_X\{x\}$ as  $\cl_X\de_X\{x\}=\de_X\{x\}$.
Conversely, since $\cl_X\{x\}$ is a closed superset of  $\de_X\{x\}$, it includes $\cl_X\de_X\{x\}$, so we have
$
\cl\nolimits_X\de\nolimits_X\{x\} \sub \cl\nolimits_X\{x\}= \de\nolimits_X\{x\}\cup\{x\}.
$
Thus if  $x$ is not in $\cl_X\de_X\{x\}$, then $\cl_X\de_X\{x\}\sub\de_X\{x\}$, hence  $\de_X\{x\}$ is closed.
\qed
\end{proof}

We now review the topological semantics for modal logics \citep{bezh:some05}.
 A \emph{topological model} $\M=(X,V)$ on a space $X$ has a valuation $V$ assigning a subset of $X$ to each propositional variable. A truth set $\M_d(\ph)$ is then defined by induction on the formation of  $\ph$ by letting $\M_d (p)=V(p)$,  interpreting the Boolean connectives by the corresponding Boolean set operations, and putting 
 $\M_d(\di\ph)=\de_X(\M_d(\ph))$, the set of limit points of $\M_d(\ph)$. Then  $\M_d(\bo\ph)$ is determined by the requirement that it be equal to $\M_d(\neg\di\neg\ph)$. This gives
 \begin{itemize}
\item $x\in \M_d(\di\ph)$ iff every punctured neighbourhood of $x$ intersects $ \M_d(\ph)$;

\item $x\in \M_d(\bo\ph)$ iff there is a punctured neighbourhood  of $x$ included in $ \M_d(\ph)$.
\end{itemize}

A formula $\ph$ is \emph{$d$-true in $\M$}, written $\M\models_d\ph$, if $\M_d(\ph)=X$; and is \emph{d-valid in space $X$}, written $X\models_d\ph$, if it is $d$-true  in all models on $X$. 
The set $\{\ph:X\models_d\ph\}$ of all formulas that are $d$-valid in $X$ is a normal logic, called the \emph{$d$-logic of $X$}. It need not be a transitive logic, because the scheme 4 is $d$-valid in $X$ iff $X$ is $T_D$ \citep{esak:weak01,esak:intu04}.

A point $x$ is \emph{isolated} in a space  $X$  if $\{x\}$ is open in $X$. If $X$ has no isolated points, it  is  \emph{crowded}  (also called \emph{dense-in-itself}). This means that every point is a limit point of $X$, i.e.\ $\de_X X=X$. Thus $X$ is crowded iff $X\models_d\di\top$.
 A subset $S$ is called \emph{crowded in $X$} if it is crowded as a subspace, i.e.\ $\de_S S=S$, which is equivalent to requiring that $S\sub\de_X S$, and hence that $\cl_X S=\de_X S$, since $\cl_X S=S\cup\de_X S$. The following is standard.
 
 \begin{lemma} \label{densecrowded}
Let $O$ and $S$ be subsets of $X$ with $O$ open. If $S$ is dense in $X$ then $O\cap S$ is dense in $O$, and if $S$ is crowded in $X$ then $O\cap S$ is crowded in $O$.
\end{lemma}
 
 \begin{proof}
If $\cl_X S=X$, then $O=O\cap \cl_X S\sub\cl_X( O\cap  S)$, showing that $O\cap S$ is dense in $O$.
If   $S$ is crowded, then $O\cap S\sub O\cap\de_X S\sub \de_X(O\cap S)$, so $O\cap S$ is crowded.
\qed
\end{proof}

We now explain a relationship between $d$-validity and frame validity.
Let $\F=(W,R)$ be a transitive frame. There is an associated \emph{Alexandrov topology} on $W$ in which the open subsets $O$ are  those that are  \emph{up-sets} under $R$, i.e.\ if $w\in O$ and $wRv$ then $v\in O$.
Call the resulting topological space $W_R$. Let $R^*=R\cup\{(w,w):w\in W\}$ be the reflexive closure of $R$. Then $R^*$
is a \emph{quasi-order}, i.e.\  is reflexive and transitive, and
the topology of $W_{R}$ has as a basis the sets $R^*(w)$ for all $w\in W$, where
$
R^*(w)=\{v:wR^* v\}=\{w\}\cup R(w).
$

A \emph{$d$-morphism} from a space $X$ to $\F$ is a function $f:X\to W$ that has the following properties:
\begin{enumerate}[\rm(i)]
\item 
$f$ is a continuous and open function from $X$ to the space $W_{R}$.
\item
If $w\in W$ is  $R$-reflexive, then  the preimage $f^{-1}\{w\}$ is crowded in $X$.
\item
If  $w$ is $R$-irreflexive, then  $f^{-1}\{w\}$ is a discrete subspace of $X$, i.e.\ each point of  $f^{-1}\{w\}$ is isolated in  $f^{-1}\{w\}$, or equivalently $f^{-1}\{w\}\cap\de_X f^{-1}\{w\}=\emptyset$.
\end{enumerate}
In (i), $f$ is continuous when the $f$-preimage of any open subset of $W_R$ is open in $X$, while $f$ is open when the $f$-image of any open subset of $X$ is open in $W_R$.
The importance of this kind of morphism is that a surjective $d$-morphism preserves $d$-validity as frame validity, in the following sense.

\begin{theorem} [{\citealt[Cor.~2.9]{bezh:some05}}]  
\label{dmorph}
If there exists a $d$-morphism from $X$ \emph{onto} $\F$, then for any formula $\ph$,\enspace
$X\models_d\ph$ implies $\F\models\ph$.
\qed
\end{theorem}

 The interpretation of $\di$ by $\de$ is sometimes called \emph{$d$-semantics} \citep{bezh:some05}. It has
 $\M_d(\di^*\ph)=\cl_X(\M_d(\ph))$ and  $\M_d(\bo^*\ph)=\int_X(\M_d(\ph))$,  because $\cl_X Y=Y\cup\de_X Y$. 
 By contrast, \emph{$C$-semantics} interprets $\di$ as $\cl$, defining truth sets $\M_C(\ph)$ inductively  in a topological model $\M$ by putting $\M_C(\di\ph)=\cl_X(\M_C(\ph))$ and $\M_C(\bo\ph)=\int_X(\M_C(\ph))$. A formula $\ph$ is \emph{$C$-valid in $X$}, written $X\models_C\ph$, iff 
$\M_C(\ph)=X$ for all models $\M$ on $X$.
 In $C$-semantics there is no distinction in interpretation between $\di$ and $\di^*$, or between $\bo$ and $\bo^*$.
 
 A space is \emph{scattered} if each of its non-empty subspaces has an isolated point, i.e.\ no non-empty subset is crowded. This condition $d$-validates the L\"ob axiom, and hence $d$-validates the logic K4$\C_0$, since it is equal to the G\"odel-L\"ob logic. In fact K4$\C_0$ is characterised by $d$-validity in all scattered spaces, a result due to \cite{esak:diag81}. It can be readily explained via the relational semantics. If $\ph$ is a non-theorem of K4$\C_0$, then $\ph$ fails to be valid in some frame $(W,R)$ with $W$ finite and $R$ irreflexive and transitive. In such a frame, $R^{-1}Y$ is the derived set of $Y$ in the Alexandrov space $W_R$, for any $Y\sub W$. This implies that the relational semantics on $(W,R)$ agrees with the $d$-semantics on $W_R$, in the sense that  a formula is true at a point $w$ in a relational model $(W,R,V)$ iff it is $d$-true at $w$ in the topological model $(W_R,V)$. Hence $\ph$ is not $d$-valid in the space $W_R$. But $W_R$ is scattered, since for any non-empty $Y\sub W$ there is an $R$-maximal element $w\in Y$, i.e.\ $wRv$ implies $v\notin Y$, hence $R^*(w)$ is an open neighbourhood of $w$ in $W_R$ that contains no member of $Y$ other than $w$, making $w$ isolated in $Y$.
 
 From now on we focus on the logics K4$\C_n$ with $n\geq 1$.

\section{Hereditary Irresolvability}

A \emph{partition} of a space $X$ is, as usual, a collection of non-empty subsets of $X$, called the \emph{cells}, that are pairwise disjoint and whose union is $X$. It is a \emph{$k$-partition}, where $k$ is a positive integer, if it has exactly $k$ cells. A \emph{dense partition} is one  for which each cell is  dense in $X$. A \emph{crowded partition} is one whose  cells are  crowded.

For $k\geq 2$, a  space is called \emph{$k$-resolvable} if it has $k$ pairwise disjoint non-empty dense subsets.  Since any superset of a dense set is dense, $k$-resolvability is equivalent to $X$ having a dense $k$-partition. $X$ is \emph{$k$-irresolvable} if it is not $k$-resolvable. It is  \emph{hereditarily $k$-irresolvable}, which may be abbreviated to $k$-HI, if every non-empty subspace of $X$ is $k$-irresolvable. 
Note that if $k\leq n$, then an $n$-resolvable space is also $k$-resolvable, since we can amalgamate cells of a dense partition to form new dense partitions with fewer cells. Hence if $X$ is $k$-HI, then it is also $n$-HI.

$k$-resolvability was defined by \cite{cede:maxi64} with the extra requirement that each cell of a dense $k$-partition should intersect each non-empty open set in at least $k$ points. That requirement was dropped by later authors, including \cite{elki:reso69} and \cite{ecke:reso97}, with the latter defining the $k$-HI notion.

The prefix $k$- is usually omitted when $k=2$. Thus a space is \emph{resolvable} if it has a disjoint pair of non-empty dense subsets, and is  \emph{hereditarily irresolvable}, or HI, if it has no non-empty subspace that is resolvable.

It is known that any HI space is T$_D$. For convenience we repeat an explanation of this from \citep{gold:moda19}.
In general  $\cl_X\{x\}=\{x\}\cup\de_X\{x\}$ with $x\notin \de_X\{x\}$ and $\{x\}$  dense in $\cl_X\{x\}$, while
 $\cl_X\de_X\{x\}\sub \cl_X\{x\}$.  
 But if $X$ is HI, then $\cl_X\{x\}$ is irresolvable, so $\de_X\{x\}$ cannot be dense in $\cl_X\{x\}$, hence  $\cl_X\de_X\{x\}$ can only be  $\de_X\{x\}$, i.e.\  $\de_X\{x\}$ is closed.

On the other hand, a $k$-HI space need not be T$_D$ when $k>2$. For instance, we saw that a two-element indiscrete space is not T$_D$, but since it has no 3-partition it is $k$-HI for every $k>2$.

It is also known that every scattered space is HI.  For, in a scattered space any non-empty subspace $Y$ has an isolated point which will belong to one cell of any 2-partition of $Y$ and prevent the other cell from being dense, hence prevent $Y$ from being resolvable. It follows that every scattered space is T$_D$. This is a topological manifestation of the celebrated proof-theoretic fact that the transitivity axiom 4 is derivable from L\"ob's axiom over K (see \citealt[p.~11]{bool:logi93}).

In \citep{gold:moda19}, the following results were proved for all $n\geq 1$, where S4$\C_n$ is the smallest normal extension of K4$\C_n$ that includes the scheme $\bo\ph\to\ph$, or equivalently $\ph\to\di\ph$.

\begin{enumerate}
\item 
A  space $X$ has $X\models_d\C_n$  iff $X\models_C\C_n$
iff $X$ is hereditarily $n+1$-irresolvable.

\item
If $(W,R)$ is a finite quasi-order, then it has circumference at most $n$ iff the space $W_R$ is hereditarily $n+1$-irresolvable.

\item
S4$\C_n$ is characterised by $C$-validity in all  hereditarily $n+1$-irresolvable spaces, i.e.\
a formula is a theorem of S4$\C_n$ iff it is $C$-valid in all  hereditarily $n+1$-irresolvable spaces. Moreover, S4$\C_n$ is characterised by $C$-validity in all \emph{finite} hereditarily $n+1$-irresolvable spaces.

\item
K4$\C_n$ is not characterised by $d$-validity in any class of finite spaces.
\end{enumerate}
The reason for the last result is that every finite space that $d$-validates K4 is scattered and so
$d$-validates the  L\"ob axiom, which is not a theorem of K4$\C_n$ when $n\geq 1$.

Using the first result listed above we can  infer that  K4$\C_n$ is sound for 
$d$-validity in all hereditarily $n+1$-irresolvable T$_D$ spaces.
The main result of this paper is that, conversely,  K4$\C_n$ is complete for $d$-validity in all hereditarily $n+1$-irreducible T$_D$ spaces (albeit not for $d$-validity in all the finite ones). To indicate how this will be proved, note that by Theorem \ref{K4Cnfmp} we have that K4$\C_n$ is complete for  $d$-validity in all finite frames that validate K4$\C_n$. So to show that  K4$\C_n$ is complete for $d$-validity in some class of spaces, it suffices by Theorem \ref{dmorph}  to show that every finite  K4$\C_n$-frame is a $d$-morphic image of some  space in that class.
For $n=1$ this was done by \cite{gabe:topo04} (see also \citealt{bezh:k4gr10}), proving that K4$\C_1$ (in the form K4Grz$_\Box$) is the $d$-logic of HI-spaces by showing that any finite  K4$\C_1$-frame is a $d$-morphic image of  an HI space. A finite  K4$\C_1$-frame has only singleton clusters, and each non-degenerate one was replaced by an HI space to construct the desired HI preimage. We will now generalise this construction to make it work for all $n> 1$ as well.

Suppose that $\F=(W,R)$ is finite and transitive.
Let $\CC$ be the set of $R$-clusters of $\F$. 
We define a collection $\{X_C:C\in\CC\}$ of spaces, with each $X_C$ having a partition  
$\{X_w:w\in C\}$ indexed by $C$.
If $C=\{w\}$ is a degenerate cluster, put $X_C=X_w=\{w\}$ as a one-point space. For $C$ non-degenerate, with
 $C=\{w_1,\dots,w_k\}$ for some positive integer $k$, take $X_C$ to be 
a copy of a  space that has a crowded dense $k$-partition.  Label the cells of that partition  $X_{w_1},\dots,X_{w_k}$.
We take  $X_C$ to be disjoint from $X_{C'}$ whenever $C\ne C'$ (replacing spaces by homeomorphic copies where necessary to achieve this). That completes the definition of the $X_C$'s and the $X_w$'s.

Now let $X_\F=\bigcup\{X_C:C\in\CC\}$, and
define a surjective map $f:X_\F\to W$ by putting $f(x)=w$ iff $x\in X_w$. This entails that  $f^{-1}\{w\}=X_w$ and   $f^{-1}C=X_C$.

For $C,C'\in\CC$, write $C R\up C'$    if $CRC'$ but not $C'RC$, i.e.\ $C'$ is a strict $R$-successor of $C$. Define a subset $O\sub X_\F$ to be \emph{open} iff for all $C\in\CC$, $O\cap X_C$ is open in $X_C$ and 
\begin{equation*}
\text{
if $O\cap X_C\ne\emptyset$, then
for all $C'$ such that $CR\up C'$,\, $X_{C'}\sub O$. }
\end{equation*}
It is readily checked that these open sets form  a topology on $X_\F$. If $B$ is an open subset of $X_C$, then
$O_B=B\cup\bigcup\{X_{C'}:C R\up C'\}$ is an open subset of $X_\F$ (this uses transitivity of $R\up$), with $O_B\cap X_C=B$. It follows that $X_C$ is a subspace of $X_\F$, i.e.\ the original topology of $X_C$ is identical to the subspace topology on the underlying set of $X_C$ inherited from the topology of  $X_\F$.

\begin{lemma} \label{lemdmorph}
$f$ is a $d$-morphism from $X_\F$ onto $\F$.
\end{lemma}

\begin{proof}
To show $f$ is continuous it is enough to show that the preimage $f\inv R^*\{w\}$ of any basic open subset of $W_{R^*}$ is open in $X_\F$. If $C$ is the $R$-cluster of $w$, then $R^*\{w\}=C\cup\bigcup\{{C'}:C R\up C'\}$, so
$$\textstyle
f\inv R^*\{w\}= f\inv C\cup \bigcup\{f\inv {C'}:C R\up C'\}
=X_C\cup \bigcup\{X_{C'}:C R\up C'\},
$$
which is indeed open in $X_\F$.

To show that $f$ is an open map, we must show that if $O$ is an open subset of $X_\F$, then $f(O)$ is open in $W_{R^*}$, i.e.\ is  an $R$-up-set.
So, suppose  $w\in f(O)$ and $wRv$. We want $v\in f(O)$.
Let $C$ be the cluster of $w$. We have $w=f(x)$ for some $x\in O\cap X_C$. If $v\in C$, then $C$ is non-degenerate and  $X_v$ is dense in $X_C$, so as $ O\cap X_C$ is open in $X_C$, there is some $y\in  X_v\cap O$. Then $v=f(y)\in f(O)$.
If however $v\notin C$, then the cluster $C'$ of $v$ has $CR\up C'$, hence  $X_{C'}\sub O$. Taking any $y\in X_v\sub X_{C'}$ gives   $v=f(y)\in f(O)$ again. That completes the proof that $f(O)$ is an $R$-up-set.

If $w\in C$ is reflexive, then $f\inv\{w\}=X_w$ is crowded in $X_C$, i.e.\ $f\inv\{w\}\sub \de_{X_C}f\inv\{w\}$. 
But $\de_{X_C}f\inv\{w\}\sub \de_{X_\F}f\inv\{w\}$,
since $X_C$ is a subspace of $X_\F$, so $f\inv\{w\}=X_w$ is crowded in $X_\F$.

Finally,  if $w$ is irreflexive, then  $f\inv\{w\}=\{w\}$ is discrete in $X_\F$.
\qed
\end{proof}

\begin{lemma}    \label{lemTD}
If $X_C$ is T$_D$ for all $C\in\CC$, then $X_\F$ is T$_D$.
\end{lemma}

\begin{proof}
By Lemma \ref{TDchar}, a space $X$ is T$_D$ iff it has $x\notin \cl_X\de_X\{x\}$ in general. If $x\in X_\F$, then $x\in X_C$ for some $C$. If $X_C$ is T$_D$, then there is an open neighbourhood $O$ of $x$ in $X_C$ that is disjoint from $\de_{X_C}\{x\}$. As
 $\de_{X_C}\{x\}=X_C\cap  \de_{X_\F}\{x\}$, $O$ is  disjoint from $\de_{X_\F}\{x\}$.
Let $O'= \bigcup\{X_{C'}:C R\up C'\}$. Then $O'$ is $X_\F$-open with $x\notin O'$, so no point of $O'$ is a limit point of $\{x\}$ in $X_\F$. Hence $O\cup O'$ is an $X_\F$-open neighbourhood of $x$ that is disjoint from $\de_{X_\F}\{x\}$, showing that  $x\notin \cl_{X_\F}\de_{X_\F}\{x\}$.
\qed
\end{proof}
Note that this result need not hold with T$_1$ in place of T$_D$. $X_\F$ need not be T$_1$ even when every $X_C$ is. For if  $CR\up C'$ with $x\in X_C$ and $y\in X_{C'}$, then every open neighbourhood of $x$ in $X_\F$ contains $y$, so $x\in\cl_{X_\F}\{y\}-\{y\}$, showing that $\{y\}$ is not closed. 

\begin{lemma}  \label{lemnHI}
If $X_C$ is $n$-HI for all $C\in\CC$, then $X_\F$ is $n$-HI.
\end{lemma}

\begin{proof}
If $X_\F$ is not $n$-HI, then it has some non-empty subspace $Y$ that  has $n$ pairwise disjoint subsets $S_1,\dots,S_n$ that are each dense in $Y$, i.e.\ $Y\sub\cl_{X_\F} S_i$.
Since $\CC$ is finite and $R\up$ is antisymmetric, there must be a $C\in\CC$ such that $X_C$ intersects $Y$ and $C$ is $R\up$-maximal with this property. Thus $X_C\cap Y\ne\emptyset$ but if $CR\up C'$ then $X_{C'}\cap Y=\emptyset$. Then putting  
$O=X_C\cup \bigcup\{X_{C'}:C R\up C'\}$ gives $O\cap Y=X_C\cap Y\ne\emptyset$.

Now $O$ is $X_\F$-open, so $O\cap Y$ is a non-empty $Y$-open set, hence it intersects the sets $S_i$ as they are dense in $Y$.  Thus the sets $\{O\cap S_i:1\leq i\leq n\}$ are pairwise disjoint and non-empty. They are also dense in $X_C\cap Y$, as
$$
X_C\cap Y=O\cap Y\sub O\cap \cl\nolimits_{X_\F} S_i\sub \cl\nolimits_{X_\F} (O\cap S_i),
$$
with the last inclusion holding because $O$ is  $X_\F$-open. This shows that $X_C\cap Y$ is an $n$-resolvable subspace of $X_C$, proving that $X_C$ is not $n$-HI.
\qed
\end{proof}

  To prove that K4$\C_n$ is characterised by $d$-validity in  $n+1$-HI T$_D$ spaces, we want to show that such spaces provide $d$-morphic preimages of all finite transitive frames of circumference at most $n$. To achieve this, the work so far indicates that we need to replace non-degenerate clusters by $n+1$-HI T$_D$ spaces that have crowded dense $k$-partitions for various $k\leq n$.
So we need to show such spaces exist.

The literature contains several constructions of $n$-resolvable spaces that are not $n+1$-resolvable.  For instance, 
\cite{douw:appl93} constructs ones that are crowded, countable and regular\footnote{van Douwen's terminology is different. He calls a space \emph{$n$-irresolvable} if it has a dense partition of size $n$, but none larger.}.
The most convenient construction for our purposes is given by \cite{elki:ultr69}. To describe it, first define $E$ to be a space, based on the set $\om$ of natural numbers, for which the set of open sets is $\mathcal{U}\cup\{\emptyset\}$ where $\mathcal U$ is some non-principal ultrafilter on $\om$. This makes $E$  a \emph{door} space: every subset is either open or closed. It is crowded, as no singleton belongs to $\mathcal U$, and is T$_1$ as every co-singleton $\om-\{x\}$ does belong to $\mathcal U$. 
$E$ has the special property that \emph{the intersection of any two non-empty $E$-open sets is non-empty} (infinite actually). This implies that any 
 non-empty open set is dense in $E$.
 
 The closure properties of an ultrafilter also ensure that $E$ is HI. For if a non-empty subspace $Y$ of $E$ has a 2-partition, then either $Y$ is open and so at least one cell of the partition is open, which prevents the other cell from being dense in $Y$; or else $Y$ is closed and so both cells are closed and hence neither is dense.

Now view $\om\times\{1,\dots,n\}$ as the union of its disjoint  subsets  
$\om\times\{i\}$ for $1\leq i\leq n$.
Let $X_n$ be the space based on $\om\times\{1,\dots,n\}$ whose non-empty open sets are all the sets of the form $\bigcup_{i\leq n}(O_i\times\{i\})$ where each $O_i$ is a \emph{non-empty} open subset of $E$. This definition does satisfy the axioms of a topology because of the special property of $E$ noted above. Put $S_i=\om\times\{i\}$. Then $\{S_i:1\leq i\leq n\}$ is an $n$-partition of $X_n$ that is dense because every  non-empty $X_n$-open set intersects every $S_i$, so the cells are all dense in $X_n$. Hence $X_n$ is $n$-resolvable.

The intersection of any  non-empty $X_n$-open set  with $S_i$ is of the form $O_i\times\{i\}$ with $O_i$ open in $E$. It follows that $S_i$ as a subspace of $X_n$ is a homeomorphic copy of $E$, so inherits the topological properties of $E$, including being a door space that is HI and having all its non-empty open subsets be dense.
It also follows that the non-empty open sets of $X_n$ are all the sets of the form $\bigcup_{i\leq n}O_i'$ where each $O'_i$ is a non-empty open subset of $S_i$. 

$S_i$ inherits from $E$ the property that its non-empty open sets are infinite. This implies that each $S_i$ is crowded in $X_n$, as is $X_n$ itself.

$X_n$ is also T$_1$, since for any point $(x,i)\in X_n$ the set
$$
X_n-\{(x,i)\}=S_1 \cup\cdots\cup[ (\om-\{x\})\times\{i\}]\cup\cdots\cup S_n
$$
is open in $X_n$, so $\{(x,i)\}$ is closed.

$X_n$ is not $n+1$-resolvable. This is implied by several results in the literature,  including that of  \citet[Proposition 1]{elki:reso69}, which states that a space is $n+1$-irresolvable if it has a dense $n$-partition with each cell having the property that each of its crowded subspaces is irresolvable.  
\citet[Lemma 2]{illa:fini96} proves $n+1$-irresolvability of any space that has an $n$-partition whose cells are  \emph{openly irresolvable} (OI), meaning that every non-empty \emph{open} subspace is irresolvable. 
The most general result of this type would appear to be that of  \citet[Lemma 3.2(a)]{ecke:reso97},  proving $n+1$-irresolvability of any space that is merely the union of $n$ subspaces that are each openly irresolvable. But it is instructive and more direct here to give a proof for  $X_n$ that uses its particular structure.

\begin{lemma}
If $A$ is a dense subset of $X_n$, then there exists an $i\leq n$ such that $A\cap S_i$ is non-empty and open in $S_i$.
\end{lemma}

\begin{proof}
Let $A$ be dense. Suppose that the conclusion does not hold. Then for each $i\leq n$, if $A\cap S_i$ is non-empty then it is not open in $S_i$, so is not equal to $S_i$.  Hence its complement $S_i-(A\cap S_i)$ is non-empty,  and open in $S_i$ as $S_i$ is a door space. If $A\cap S_i=\emptyset$, then  $S_i-(A\cap S_i)$ is again non-empty and open in $S_i$. Therefore the union $\bigcup_{i\leq n}[S_i-(A\cap S_i)]$ is, by definition, a non-empty open subset of $X_n$. But this union is $X_n-A$, so that contradicts the fact that $A$ is dense.
\qed
\end{proof}

Now if  $X_n$ were $n+1$-resolvable, it would have $n+1$ subsets $A_1,\dots,A_{n+1}$ that are pairwise disjoint and dense. Then by the lemma just proved, for each $j\leq n+1$ there would be some $i\leq n$ such that $A_j\cap S_i$ is non-empty and open in $S_i$, hence is dense in $S_i$ as explained above. Hence by the pigeonhole principle there must be \emph{distinct} $j,k\leq n+1$ such that there is some $i\leq n$ with  both  subsets $A_j\cap S_i$ and  $A_k\cap S_i$  dense in $S_i$. But these subsets are disjoint, so that contradicts the irresolvability of $S_i$. Therefore $X_n$ cannot be $n+1$-resolvable.

\begin{theorem}  \label{thmYn}
For any $n\geq 1$ there exists a non-empty crowded hereditarily $n+1$-irresolvable T$_1$ space $Y_n$ 
that has a crowded dense $n$-partition.
\end{theorem}

\begin{proof}
For any $k> 1$, every $k$-irresolvable space has a non-empty open subspace that is $k$-HI, constructed as the complement of the union of all $k$-resolvable subspaces \cite[Prop.~2.1]{ecke:reso97}.
So we apply this with $k=n+1$ to the  $n+1$-irresolvable space $X_n$ just described to conclude that $X_n$ has a non-empty open subspace $Y_n$ that is $n+1$-HI.
$Y_n$  inherits the  T$_1$ condition from $X_n$ and,
since $Y_n$ is open, it inherits the crowded condition from $X_n$, and it intersects each of the dense sets $S_i$. Also each intersection $S_i'=Y_n\cap S_i$ is crowded and dense in $Y_n$, as $S_i$ is crowded and dense in $X_n$ and $Y_n$ is open. Thus $\{S_i':1\leq i\leq n\}$ is a crowded dense $n$-partition of $Y_n$.
\qed
\end{proof}

\begin{theorem}  \label{thmdmorph}
For any $n\geq 1$, every finite $\mathrm{K4}\C_n$ frame is a $d$-morphic image of an hereditarily $n+1$-irresolvable T$_D$ space.
\end{theorem}

\begin{proof}
Let $\F$ be a finite K4$\C_n$ frame. $\F$ is transitive with circumference at most $n$. We carry out the construction of the space $X_\F$ as above.

For each non-degenerate cluster $C$ of $\F$,  if $C$ has $k\geq 1$ elements, we take $X_C$ to be a copy of the T$_1$ space $Y_k$ of Theorem \ref{thmYn}, and let $\{X_w:w\in C\}$ to be the crowded dense $k$-partition of $Y_k$ provided by that theorem. Now $k\leq n$ and $Y_k$ is $k+1$-HI, so it is $n+1$-HI. Also any singleton subspace is $n+1$-HI, so we see that every subspace $X_C$ of $X_\F$ is $n+1$-HI. Hence $X_\F$ is  $n+1$-HI by Lemma \ref{lemnHI}.

Every subspace $X_C$ of $X_\F$, including the singleton ones, is T$_1$, hence is T$_D$. So $X_\F$ is T$_D$  by Lemma \ref{lemTD}.

The $d$-morphism from $X_\F$ onto $\F$ is provided by Lemma \ref{lemdmorph}.
\qed
\end{proof}

In the case $n=1$ of this construction, all clusters of $\F$ are singletons, and $X_\F$ is obtained by replacing each non-degenerate cluster by a copy of the El'kin space $E$. This is exactly the construction of
\citep{gabe:topo04} and  \citep{bezh:k4gr10}.

\begin{theorem}
For any $n\geq 1$, the logic $\mathrm{K4}\C_n$ is characterised by $d$-validity in all spaces that are hereditarily $n+1$-irresolvable and T$_D$.
\end{theorem}

\begin{proof}
If a space $X$ is $n+1$-HI and $T_D$, then
 the $d$-logic of $X$ includes the schemes 4 and $\C_n$, so it includes  K4$\C_n$ as the smallest normal logic to include these schemes. Hence every theorem of  K4$\C_n$ is $d$-valid in $X$.

For the converse direction, if a formula $\ph$ is not a theorem of  K4$\C_n$, then by Theorem \ref{K4Cnfmp}(3) there is a finite frame $\F$ that validates K4$\C_n$ but does not validate $\ph$. By Theorem \ref{thmdmorph} $\F$ is a $d$-morphic image of some $n+1$-HI T$_D$ space $X$. Since $\F\not\models\ph$, Theorem \ref{dmorph} then gives $X\not\models_d\ph$. Thus it is not the case that $\ph$ is $d$-valid in all $n+1$-HI and $T_D$ spaces.
\qed
\end{proof}

As already noted, the case $n=1$ of this result was given in \citep{gabe:topo04}, and in that case  the T$_D$ condition  is redundant, as hereditarily irresolvable spaces are always T$_D$.

\section{Some Extensions of  K4$\C_n$}

The D-axiom $\di\top$ is $d$-valid in a space $X$ iff $X=\de_X X$, i.e.\ iff $X$ is crowded. In general a space of the type $X_\F$ need not be crowded, for if $C$ is a degenerate final cluster of $\F$, then $X_C$ is an open singleton containing an isolated point of $X_\F$.
But we have shown in \citep[\S 7]{gold:moda19} that  K4D$\C_n$ is characterised by validity in all finite transitive frames that have circumference at most $n$ and all final clusters \emph{non-degenerate}. If  $\F$ is such a frame, and $C'$ is any final cluster of $\F$, then $X_{C'}$ is a crowded space of the type given by Theorem \ref{thmYn}. Now any open neighbourhood of a point $x$ in $X_\F$ includes
an open set of the form 
$O_B=B\cup\bigcup\{X_{C'}:C  R\up C'\}$,
where $B$ is an open neighbourhood of $x$ in some subspace $X_C$. If $C$ is final, then $O_B=B$ and $X_C$ is crowded, so $O_B$ contains points other than $x$. If $C$ is not final then there is a final $C'$ with $C R\up C'$, so $O_B$ includes $X_{C'}$, which consists of points distinct from $x$. Thus $x$ is not isolated, showing that $X_\F$ is crowded. This leads us to conclude

\begin{theorem}
For $n\geq 1$, K4D$\C_n$ is characterised by $d$-validity in all crowded T$_D$ spaces
that are hereditarily $n+1$-irresolvable.
\qed
\end{theorem}

At the opposite extreme are  logics containing  the constant formula 
\begin{equation*} 
\text{E}: \quad   \bo\bot\lor\di\bo\bot.
\end{equation*}
This is $d$-valid in a space iff it is \emph{densely discrete}\footnote{This property was previously called \emph{weakly scattered}. The change of terminology was made, with explanation, in \cite[Definition 2.1]{bezh:tree20}.}, meaning that the set of isolated points is dense in the space.  The set of isolated points is $X-\de_X X$, so $X$ is densely discrete iff 
$ \cl_X(X-\de_X X)=X$, i.e.\
$$
(X-\de\nolimits_X X)\cup \de\nolimits_X(X-\de\nolimits_X X)=X.
$$
This equation expresses the $d$-validity of $\neg\di\top\lor \di\neg\di\top$, which is equivalent to E
(see  \citealt[proof of Theorem 4.28]{gabe:topo04}).

K4E$\C_n$ was  shown in \citep[\S 7]{gold:moda19} to be characterised by validity in all finite transitive  frames that have circumference at most $n$ and all final clusters \emph{degenerate}.  If  $\F$ is such a frame, and $C'$ is any final cluster of $\F$, then $X_{C'}$ is an open singleton, as noted above. If a point $x$ of $X_\F$ belongs to $X_C$ where $C$ is not final in $\F$, then there is a final $C'$ with $CR\up C'$, so any open neighbourhood of $x$ will include $X_{C'}$ and hence contain an isolated point. This shows that the isolated points are dense in $X_\F$, and leads to
\begin{theorem}
for $n\geq 1$, K4E$\C_n$ is characterised by $d$-validity in all densely discrete T$_D$ spaces
that are hereditarily $n+1$-irresolvable.
\qed
\end{theorem}

In Theorem \ref{thmdmorph} we could have replaced every non-degenerate cluster $C$ by a copy of the same space $Y_n$, since its crowded dense $n$-partition can be converted into a crowded dense $k$-partition for any $k<n$ by amalgamating cells. But allowing $X_C$ to vary with the size of $C$ gives more flexibility in defining spaces. This is well illustrated in the case of logics that include the well-studied  McKinsey axiom M, 
often stated as $\bo\di\ph\to\di\bo\ph$.
We use the equivalent forms
$\di(\bo\ph\lor\bo\neg\ph)$ 
and $\di\bo\ph\lor\di\bo\neg\ph$.

It follows from \cite[Prop.~2.1]{bezh:scat03} that in $C$-semantics, M defines the class of openly irresolvable (OI) spaces (recall that these are the spaces in which every non-empty \emph{open} subspace is irresolvable). Equivalently, in $d$-semantics, the scheme $\di^*(\bo^*\ph\lor\bo^*\neg\ph)$  defines the class of OI spaces. 
We give a direct proof of this.

\begin{lemma}  \label{OIdM}
A space $X$ is openly irresolvable iff $X\models_d \di^*(\bo^*\ph\lor\bo^*\neg\ph)$ for all $\ph$.
\end{lemma}

\begin{proof}
Suppose $X\not \models_d \di^*(\bo^*\ph\lor\bo^*\neg\ph)$ for some $\ph$. Then there is a model $\M$ on $X$ and a point of $X$ that is not a closure point of $\M_d((\bo^*\ph\lor\bo^*\neg\ph)$, and so has
an  open neighbourhood $U$ disjoint from this $d$-truth set. Let $A=\M_d(\ph)$. Then $U$ is disjoint from $\int_X A$ and from $\int_X (X-A)$, hence $U$ is included in $\cl_X A$ and in $\cl_X (X-A)$. Thus $U\cap A$ and $U\cap (X-A)$ are dense subsets of the non-empty open $U$, showing that $U$ is resolvable, and so $X$ is not OI.

Conversely, assume $X$ is not OI, so has a non-empty open subset $U$ which has a subset $A$ such that $A$ and $U-A$ are dense in $U$. Hence  $U$ is included in $\cl_X A$ and in $\cl_X (U-A)\sub \cl_X (X-A)$, so is 
disjoint from $\int_X A$ and from $\int_X (X-A)$. Take  a model $\M$ on $X$ with $A=\M_d(p)$ for some variable $p$. 
Then $U$ is disjoint from $\M_d((\bo^* p\lor\bo^*\neg p)$, so $\di^*(\bo^* p\lor\bo^*\neg p)$ is $d$-false in $\M$ at any member of $U$, hence is not $d$-valid in $X$.
\qed
\end{proof}

We now explore criteria for the $d$-validity of M itself.

\begin{theorem}  \label{XdM}
If $X$ is crowded and OI, then $X\models_d \mathrm{M}$.
\end{theorem}

\begin{proof}
Let $X$ be crowded and OI. Take any model $\M$ on $X$, any formula $\ph$, and let 
$S=\M_d(\bo^*\ph\lor\bo^*\neg\ph)$. 
Then $S$ is open, as the union of two interiors, so as $X$ is crowded, it follows that $S$ is crowded, hence $\cl_X S=\de_X S$. Using this, and the fact that $\M_d(\bo^*\psi)\sub\M_d(\bo\psi)$ for any $\psi$, we deduce that
\begin{equation*}
\M_d(\di^*(\bo\nolimits^*\ph\lor\bo\nolimits^*\neg\ph))= \M_d(\di(\bo\nolimits^*\ph\lor\bo\nolimits^*\neg\ph))
\sub \M_d(\di(\bo\ph\lor\bo\neg\ph)).
\end{equation*}
As $X$ is OI, 
$ \di^*(\bo^*\ph\lor\bo^*\neg\ph)$ is  $d$-true in $\M$ by Lemma \ref{OIdM}.
Therefore by the above inclusion, $\di(\bo\ph\lor\bo\neg\ph)$ is $d$-true in $\M$.

This shows that scheme M is $d$-valid in $X$.
\qed
\end{proof}

The converse of this result does not hold.
For instance, a two-element indiscrete space $X=\{x,y\}$  is resolvable, as $\{x\}$ and $\{y\}$ are dense, so $X$ is not OI, 
but it  $d$-validates M. The latter is so because the operation $\de_X$ interchanges $\{x\}$ and $\{y\}$ and leaves $X$ and $\emptyset$ fixed, from which it follows that $\di\ph\land\di\neg\ph$ is $d$-false, hence $\bo\ph\lor\bo\neg\ph$ is $d$-true, at both points in any model on $X$. So $X$  $d$-validates M.

What does hold is that in $d$-semantics, M defines the class of crowded openly irresolvable spaces \emph{ within the class of T$_D$ spaces}.

\begin{lemma} \label{Odash}
Let $X$ be crowded and T$_D$.
\begin{enumerate}[\rm 1.]
\item
Any open neighbourhood $O$ of a point $x$ in $X$ includes an  open neighbourhood $O'$ of $x$ such that $O'-\{x\}$ is non-empty and open.
\item
If $O$ is an open set in $X$, then $O\sub \cl_X S$ implies  $O\sub \de_X S$, for any $S\sub X$.
\item
$\int_X\cl_X S=\int_X\de_X S$ for any $S\sub X$.
\item
For any model $\M$ on $X$ and formula  $\ph$,
$$
\M_d( \di^*(\bo\nolimits^*\ph\lor\bo\nolimits^*\neg\ph))= \M_d(\di(\bo\ph\lor\bo\neg\ph)).
$$
\end{enumerate}
\end{lemma}

\begin{proof}
1. Let  $x\in O$ with $O$ open. Put $O'=O-\de_X\{x\}$.
Since $x\notin\de_X\{x\}$, the set $O'$ contains $x$, and is open because $\de_X\{x\}$ is closed by the T$_D$ condition.  Then $O'-\{x\}$ is non-empty, since $x$ is not isolated as $X$ is crowded. Also 
$
O'-\{x\}=O-(\de\nolimits_X\{x\}\cup\{x\})=O-\cl\nolimits_X\{x\}
$
which is open.

2.  Let $O\sub \cl_X S$ and $x\in O$. If $U$ is any open neighbourhood of $x$, then so is $O\cap U$, hence by part 1 there is a non-empty open set $O_1\sub O\cap U$ with $x\notin O_1$.
Now $O\cap S$ is dense in $O$, as $O=O\cap  \cl_X S\sub\cl_X(O\cap S)$. Therefore as $O_1$ is open in $O$, it intersects $O\cap S$. As $O_1\sub U-\{x\}$, we get that $ U-\{x\}$ intersects $S$. This proves that $x$ is a limit point of $S$, as required.

3. Putting $O=\int_X\cl_X S$ in 2, we get that $\int_X\cl_X S$ is an open subset of $\de_X S$, hence is a subset of 
$\int_X\de_X S$. Conversely, $\int_X\de_X S\sub\int_X\cl_X S$ as $\de_X S\sub \cl_X S$.

4.
It was shown in the proof of Theorem \ref{XdM}, just using the fact that $X$ is crowded, that the left truth set is included in the right one. For the reverse inclusion, working in the model $\M$, suppose
$\di(\bo\ph\lor\bo\neg\ph)$ is true at some point $x$. Then so is $\di\bo\ph\lor\di\bo\neg\ph$, hence so is one of $\di\bo\ph$ and $\di\bo\neg\ph$. If $\di\bo\ph$ is true at $x$, then so is $\di^*\bo\ph$, hence  $\bo^*\di\neg\ph$ is false at $x$.
 By part 3, $\M_d(\bo^*\di\neg\ph)=\M_d(\bo^*\di^*\neg\ph)$, so then $\bo^*\di^*\neg\ph$ is false at $x$, hence
 $\di^*\bo^*\ph$ is true at $x$.
 
 Similarly, if $\di\bo\neg\ph$ is true at $x$, then so is  $\di^*\bo^*\neg\ph$.
Since $\di\bo\ph\lor\di\bo\neg\ph$ is true at $x$, so then is $\di^*\bo\nolimits^*\ph\lor\di^*\bo\nolimits^*\neg\ph$, hence 
so is 
$\di^*(\bo\nolimits^*\ph\lor\bo\nolimits^*\neg\ph)$,
as required to prove the inclusion from right to left.
\qed
\end{proof}

\begin{theorem}\label{TDsound}
If $X$ is a T$_D$ space and $X\models_d\mathrm{M}$, then
$X$ is crowded and openly irresolvable.
\end{theorem}

\begin{proof}
Let $X$ be T$_D$ space and $X\models_d\mathrm{M}$. Then $X$ $d$-validates $\di(\bo\top\lor\bo\neg\top)$. But this formula $d$-defines $\de_X X$ in any model on $X$, so  $\de_X X=X$, i.e.\ $X$ is crowded.

We now have that $X$ is crowded and T$_D$, and any formula of the form $\di(\bo\ph\lor\bo\neg\ph)$ is $d$-valid on $X$, i.e.\ $d$-true in all models on $X$. But then by Lemma \ref{Odash}.4, $\di^*(\bo\nolimits^*\ph\lor\bo\nolimits^*\neg\ph)$ is $d$-valid in $X$. Hence $X$ is OI by Lemma \ref{OIdM}.
\qed
\end{proof}

Theorems \ref{XdM} and \ref{TDsound} combine to give

\begin{corollary}
If  $X$ is T$_D$,  then $X\models_d\mathrm{M}$ iff
$X$ is crowded and openly irresolvable.
Hence
if\/ $X$ is T$_D$ and crowded, then $X\models_d\mathrm{M}$ iff
$X$ is  openly irresolvable iff $X\models_C\mathrm{M}$.
\qed
\end{corollary}

K4M$\C_n$ was  shown in \citep[\S 7]{gold:moda19} to be characterised by validity in all finite transitive  frames that have circumference at most $n$ and all final clusters simple. If $\F$ is such a frame,  $X_\F$ is $n+1$-HI and T$_D$, as shown in the proof of Theorem \ref{thmdmorph}.
All final clusters of $\F$ are non-degenerate, which is enough to ensure that $X_\F$ is a crowded space, as explained in our discussion of K4D$\C_n$. 

\begin{lemma}
$X_\F$ is openly irresolvable.
\end{lemma}
\begin{proof}
Let $O$ be any non-empty  open subset of $X_\F$. Then $O\cap X_C\ne\emptyset$ for some cluster $C$ of $\F$. Put $B=O\cap X_C$. Then
$O_B=B\cup\bigcup\{X_{C'}:C R\up C'\}$  is a non-empty subset of $O$ that is open in $X_\F$.  If $C$ is final, then since it is a non-degenerate singleton, $X_C$ is a copy of the El'kin space $E$, and also $O_B=B\sub X_C$. 
If however $C$ is not final,  there is a final $C'$ with $CR\up C'$. Then  $X_{C'}\sub O_B$ and $X_{C'}$ is open in $X_\F$.

So in any case we see that $O$ has a non-empty open subset $O'$ (either $O_B$ or $X_{C'}$) that is included in a subspace $X'$ (either $X_C$ or $X_{C'}$) that is a copy of $E$ and hence is HI. Hence $O'$ is irresolvable.
Now if $O$ had a pair of disjoint dense subsets, then these subsets would intersect the open $O'$ in a pair of disjoint dense subsets of $O'$, contradicting irresolvability of $O'$.
Therefore $O$ is irresolvable as required.
\qed
\end{proof}
Altogether we have now shown that $X_\F$ is T$_D$, crowded, OI, and $n+1$-HI, which implies that it $d$-validates K4M$\C_n$. 
We conclude 
\begin{theorem}
For $n\geq 1$, K4M$\C_n$ is characterised by $d$-validity in all T$_D$ spaces that are crowded, 
openly irresolvable,  and hereditarily $n+1$-irresolvable.
\qed
\end{theorem}

When $n=1$, this can be simplified, since an HI space is always OI and T$_D$. Thus the logic K4M$\C_1$ is characterised
by $d$-validity in the class of all crowded HI spaces, as was  shown by \citet[Theorem 4.26]{gabe:topo04} with K4M$\C_1$ in the form 
K4MGrz$_\Box$. But the class of 
crowded HI spaces characterises  K4D$\C_1$, as shown above. Therefore
K4M$\C_1$ is identical to the ostensibly weaker K4D$\C_1$. This can also be seen quite simply from our relational completeness result for  K4D$\C_1$. In a finite  K4D$\C_1$-frame, any final cluster is a singleton by validity of $\C_1$ and is non-degenerate by validity of D, so all final clusters are simple, making the frame validate M. Thus M is a theorem of K4D$\C_1$. 

We can also deal with the logic K4M, which is characterised by finite transitive frames in which all final clusters are simple
 \cite[\S 5.3]{chag:moda97}.
 The space $X_\F$ can be constructed without assuming that $\F$ has any bound on its circumference. If $\F$ validates K4M, then $X_\F$ will be T$_D$, crowded and OI, so will $d$-validate K4M. From this we can conclude 
 \begin{theorem}
K4M is characterised by $d$-validity in all T$_D$ spaces that are crowded and
openly irresolvable.
\qed
 \end{theorem}


\begin{thebibliography}{50}
\providecommand{\natexlab}[1]{#1}
\providecommand{\url}[1]{\texttt{#1}}
\expandafter\ifx\csname urlstyle\endcsname\relax
  \providecommand{\doi}[1]{doi: #1}\else
  \providecommand{\doi}{doi: \begingroup \urlstyle{rm}\Url}\fi

\bibitem[Allwein and Dunn(1993)]{allw:krip93}
Gerard Allwein and J.~Michael Dunn.
\newblock {K}ripke models for linear logic.
\newblock \emph{The Journal of Symbolic Logic}, 58\penalty0 (2):\penalty0
  514--545, 1993.

\bibitem[Allwein and Hartonas(1993)]{allw:dual93}
Gerard Allwein and Chrysafis Hartonas.
\newblock Duality for bounded lattices.
\newblock Indiana University Logic Group, Preprint Series, {IULG}--93--25,
  1993.

\bibitem[Aull and Thron(1962)]{aull:sepa62}
C.~E. Aull and W.~J. Thron.
\newblock Separation axioms between {T}$_0$ and {T}$_1$.
\newblock \emph{Indagationes Mathematicae (Proceedings)}, 65:\penalty0 26--37,
  1962.

\bibitem[Bezhanishvili et~al.(2020)Bezhanishvili, Bezhanishvili, Lucero-Bryan,
  and van Mill]{bezh:tree20}
G.~Bezhanishvili, N.~Bezhanishvili, J.~Lucero-Bryan, and J.~van Mill.
\newblock Tree-like constructions in topology and modal logic.
\newblock \emph{Archive for Mathematical Logic}, 2020.
\newblock \url{https://doi.org/10.1007/s00153-020-00743-6}.

\bibitem[Bezhanishvili et~al.(2003)Bezhanishvili, Mines, and
  Morandi]{bezh:scat03}
Guram Bezhanishvili, Ray Mines, and Patrick~J. Morandi.
\newblock Scattered, {H}ausdorff-reducible, and hereditarily irresolvable
  spaces.
\newblock \emph{Topology and its Applications}, 132:\penalty0 291--306, 2003.

\bibitem[Bezhanishvili et~al.(2005)Bezhanishvili, Esakia, and
  Gabelaia]{bezh:some05}
Guram Bezhanishvili, Leo Esakia, and David Gabelaia.
\newblock Some results on modal axiomatization and definability for topological
  spaces.
\newblock \emph{Studia Logica}, 81:\penalty0 325--355, 2005.

\bibitem[Bezhanishvili et~al.(2010)Bezhanishvili, Esakia, and
  Gabelaia]{bezh:k4gr10}
Guram Bezhanishvili, Leo Esakia, and David Gabelaia.
\newblock K4.{G}rz and hereditarily irresolvable spaces.
\newblock In Solomon Feferman and Wilfried Sieg, editors, \emph{Proofs,
  Categories and Computations.Essays in Honor of {G}rigori {M}ints}, pages
  61--69. College Publications, 2010.

\bibitem[Blackburn et~al.(2001)Blackburn, {de }Rijke, and Venema]{blac:moda01}
Patrick Blackburn, Maarten {de }Rijke, and Yde Venema.
\newblock \emph{Modal Logic}.
\newblock Cambridge University Press, 2001.

\bibitem[Boolos(1993)]{bool:logi93}
George Boolos.
\newblock \emph{The Logic of Provability}.
\newblock Cambridge University Press, 1993.

\bibitem[Ceder(1964)]{cede:maxi64}
J.~G. Ceder.
\newblock On maximally resolvable spaces.
\newblock \emph{Fundamenta Mathematicae}, 55:\penalty0 87--93, 1964.

\bibitem[Chagrov and Zakharyaschev(1997)]{chag:moda97}
Alexander Chagrov and Michael Zakharyaschev.
\newblock \emph{Modal Logic}.
\newblock Oxford University Press, 1997.

\bibitem[Craig and Haviar(2014)]{crai:reco14}
Andrew Craig and Miroslav Haviar.
\newblock Reconciliation of approaches to the construction of canonical
  extensions of bounded lattices.
\newblock \emph{Mathematica Slovaca}, 64\penalty0 (6):\penalty0 1335--1356,
  2014.

\bibitem[Dzik et~al.(2006)Dzik, Orlowska, and van Alten]{dzik:rela06}
Wojciech Dzik, Ewa Orlowska, and Clint~J. van Alten.
\newblock Relational representation theorems for lattices with negations: A
  survey.
\newblock In Harrie C.~M. de~Swart, Ewa Orlowska, Gunther Schmidt, and Marc
  Roubens, editors, \emph{Theory and Applications of Relational Structures as
  Knowledge Instruments II}, volume 4342 of \emph{Lecture Notes in Computer
  Science}, pages 245--266. Springer, 2006.

\bibitem[Eckertson(1997)]{ecke:reso97}
Frederick~W. Eckertson.
\newblock Resolvable, not maximally resolvable spaces.
\newblock \emph{Topology and its Applications}, 79:\penalty0 1--11, 1997.

\bibitem[El'kin(1969{\natexlab{a}})]{elki:reso69}
A.~G. El'kin.
\newblock Resolvable spaces which are not maximally resolvable.
\newblock \emph{Moscow University Mathematics Bulletin}, 24:\penalty0 116--118,
  1969{\natexlab{a}}.

\bibitem[El'kin(1969{\natexlab{b}})]{elki:ultr69}
A.~G. El'kin.
\newblock Ultrafilters and undecomposable spaces.
\newblock \emph{Moscow University Mathematics Bulletin}, 24:\penalty0 37--40,
  1969{\natexlab{b}}.

\bibitem[Engelking(1977)]{enge:gene77}
R.~Engelking.
\newblock \emph{General Topology}.
\newblock PWN --- Polish Scientific Publishers, Warsaw, 1977.

\bibitem[Esakia(1981)]{esak:diag81}
L.~L. Esakia.
\newblock Diagonal constructions, {L}{\"o}bÕs formula and {C}antorÕs scattered
  spaces.
\newblock In \emph{Studies in Logic and Semantics}, pages 128--143.
  Metsniereba, Tbilisi, 1981.
\newblock In Russian.

\bibitem[Esakia(2001)]{esak:weak01}
L.~L. Esakia.
\newblock Weak transitivity---a restitution.
\newblock \emph{Logical Investigations}, 8:\penalty0 244--255, 2001.
\newblock In Russian.

\bibitem[Esakia(2004)]{esak:intu04}
Leo Esakia.
\newblock Intuitionistic logic and modality via topology.
\newblock \emph{Annals of Pure and Applied Logic}, 127:\penalty0 155--170,
  2004.

\bibitem[Fine(1985)]{fine:logi85}
Kit Fine.
\newblock Logics containing {K}4. {P}art {II}.
\newblock \emph{The Journal of Symbolic Logic}, 50\penalty0 (3):\penalty0
  619--651, 1985.

\bibitem[Gabelaia(2004)]{gabe:topo04}
David Gabelaia.
\newblock \emph{Topological, Algebraic and Spatio-Temporal Semantics for
  Multi-Dimensional Modal Logics}.
\newblock PhD thesis, King's College London, 2004.

\bibitem[Gehrke(2018)]{gehr:cano18}
Mai Gehrke.
\newblock Canonical extensions: an algebraic approach to {S}tone duality.
\newblock \emph{Algebra Universalis}, 79\penalty0 (Article 63), 2018.
\newblock \url{https://doi.org/10.1007/s00012-018-0544-6}.

\bibitem[Gehrke and Harding(2001)]{gehr:boun01}
Mai Gehrke and John Harding.
\newblock Bounded lattice expansions.
\newblock \emph{Journal of Algebra}, 239:\penalty0 345--371, 2001.

\bibitem[Goldblatt(2018)]{gold:cano18}
Robert Goldblatt.
\newblock Canonical extensions and ultraproducts of polarities.
\newblock \emph{Algebra Universalis}, 79\penalty0 (Article 80), 2018.
\newblock \url{https://doi.org/10.1007/s00012-018-0562-4}.

\bibitem[Goldblatt(2019)]{gold:moda19}
Robert Goldblatt.
\newblock Modal logics that bound the circumference of transitive frames.
\newblock \emph{\url{arXiv:1905.11617}}, 2019.

\bibitem[Goldblatt(2020)]{gold:morp20}
Robert Goldblatt.
\newblock Morphisms and duality for polarities and lattices with operators.
\newblock \emph{Journal of Applied Logics -- IfCoLog Journal of Logics and
  their Applications}, to appear. Also \url{arXiv:1902.09783}, 2020.

\bibitem[Hartonas and Dunn(1997)]{hart:ston97}
C.~Hartonas and J.~M. Dunn.
\newblock {S}tone duality for lattices.
\newblock \emph{Algebra Universalis}, 37:\penalty0 391--401, 1997.

\bibitem[Hartung(1993)]{hart:exte93}
G.~Hartung.
\newblock An extended duality for lattices.
\newblock In K.~Denecke and H.-J. Vogel, editors, \emph{General Algebra and
  Applications}, pages 126--142. Heldermann-Verlag, Berlin, 1993.

\bibitem[Hartung(1992)]{hart:topo92}
Gerd Hartung.
\newblock A topological representation of lattices.
\newblock \emph{Algebra Universalis}, 29:\penalty0 273--299, 1992.

\bibitem[Illanes(1996)]{illa:fini96}
Alejandro Illanes.
\newblock Finite and $\omega$-resolvability.
\newblock \emph{Proceedings of the American Mathematical Society}, 124\penalty0
  (4):\penalty0 1243--1246, 1996.

\bibitem[J{\'o}nsson and Tarski(1951)]{jons:bool51}
Bjarni J{\'o}nsson and Alfred Tarski.
\newblock {B}oolean algebras with operators, part {I}.
\newblock \emph{American Journal of Mathematics}, 73:\penalty0 891--939, 1951.

\bibitem[McKinsey(1941)]{mcki:solu41}
J.~C.~C. McKinsey.
\newblock A solution of the decision problem for the {L}ewis systems {S}2 and
  {S}4 with an application to topology.
\newblock \emph{The Journal of Symbolic Logic}, 6:\penalty0 117--134, 1941.

\bibitem[McKinsey and Tarski(1944)]{mcki:alge44}
J.~C.~C. McKinsey and Alfred Tarski.
\newblock The algebra of topology.
\newblock \emph{Annals of Mathematics}, 45:\penalty0 141--191, 1944.

\bibitem[McKinsey and Tarski(1948)]{mcki:theo48}
J.~C.~C. McKinsey and Alfred Tarski.
\newblock Some theorems about the sentential calculi of {L}ewis and {H}eyting.
\newblock \emph{The Journal of Symbolic Logic}, 13:\penalty0 1--15, 1948.

\bibitem[Plo\v{s}\v{c}ica(1995)]{plos:natu95}
Miroslav Plo\v{s}\v{c}ica.
\newblock A natural representation of bounded lattices.
\newblock \emph{Tatra Mountains Mathematical Publications}, 5:\penalty0 75--88,
  1995.

\bibitem[Priestley(1970)]{prie:repr70}
H.~A. Priestley.
\newblock Representations of distributive lattices by means of ordered {S}tone
  spaces.
\newblock \emph{Bull. London Math. Soc.}, 2:\penalty0 186--190, 1970.

\bibitem[Rescher and Urquhart(1971)]{resc:temp71}
N.~Rescher and A.~Urquhart.
\newblock \emph{Temporal Logic}.
\newblock Springer-Verlag, 1971.

\bibitem[Segerberg(1971)]{sege:essa71}
Krister Segerberg.
\newblock \emph{An Essay in Classical Modal Logic}, volume~13 of
  \emph{Filosofiska Studier}.
\newblock Uppsala Universitet, 1971.

\bibitem[Stone(1936)]{ston:repr36}
M.~H. Stone.
\newblock The theory of representations for {B}oolean algebras.
\newblock \emph{Transactions of the American Mathematical Society},
  40:\penalty0 37--111, 1936.

\bibitem[Tang(1938)]{tang:alge38}
Tsao-Chen Tang.
\newblock Algebraic postulates and a geometric interpretation for the {L}ewis
  calculus of strict implication.
\newblock \emph{Bulletin of the American Mathematical Society}, 44:\penalty0
  737--744, 1938.

\bibitem[Urquhart(1978)]{urqu:topo78}
Alasdair Urquhart.
\newblock A topological representation theory for lattices.
\newblock \emph{Algebra Universalis}, 8:\penalty0 45--58, 1978.

\bibitem[Urquhart(1981)]{urqu:deci81}
Alasdair Urquhart.
\newblock Decidability and the finite model property.
\newblock \emph{Journal of Philosophical Logic}, 10\penalty0 (3):\penalty0
  367--370, 1981.

\bibitem[Urquhart(1983)]{urqu:rele83}
Alasdair Urquhart.
\newblock Relevant implication and projective geometry.
\newblock \emph{Logique et Analyse}, 103--104:\penalty0 345--357, 1983.

\bibitem[Urquhart(1984)]{urqu:unde84}
Alasdair Urquhart.
\newblock The undecidability of entailment and relevant implication.
\newblock \emph{The Journal of Symbolic Logic}, 49\penalty0 (4):\penalty0
  1059--1073, 1984.

\bibitem[Urquhart(1993)]{urqu:fail93}
Alasdair Urquhart.
\newblock Failure of interpolation in relevant logics.
\newblock \emph{Journal of Philosophical Logic}, 22\penalty0 (5):\penalty0
  449--479, 1993.

\bibitem[Urquhart(2015)]{urqu:firs15}
Alasdair Urquhart.
\newblock First degree formulas in quantified {S}5.
\newblock \emph{The Australasian Journal of Logic}, 12\penalty0 (5):\penalty0
  204--210, 2015.
\newblock \url{https://ojs.victoria.ac.nz/ajl/article/view/470}.

\bibitem[Urquhart(2017)]{urqu:geom17}
Alasdair Urquhart.
\newblock The geometry of relevant implication.
\newblock \emph{IFCoLog Journal of Logics and their Applications}, 4\penalty0
  (3):\penalty0 591--604, 2017.
\newblock \url{http://collegepublications.co.uk/ifcolog/?00012}.

\bibitem[van Douwen(1993)]{douw:appl93}
Eric~K. van Douwen.
\newblock Applications of maximal topologies.
\newblock \emph{Topology and its Applications}, 51:\penalty0 125--139, 1993.

\bibitem[Willard(1970)]{will:gene70}
Stephen Willard.
\newblock \emph{General Topology}.
\newblock Addison-Wesley, 1970.
\newblock Dover Publications Edition 2004.

\end{thebibliography}
\end{document}